\documentclass[10pt,reqno]{amsart}
\usepackage{indentfirst, amssymb,amsmath,amsthm}
\newtheorem*{theoA}{Theorem A}
\newtheorem*{theoB}{Theorem B}
\newtheorem*{theoC}{Theorem C}
\newtheorem*{theoD}{Theorem D}
\newtheorem*{theoE}{Theorem E}

\newtheorem{theo}{Theorem}
\newtheorem{lem}{Lemma}
\newtheorem{cor}{Corollary}

\newtheorem{ex}{Example}
\newtheorem{rem}{Remark}

\newcommand{\ol}{\overline}
\newcommand{\be}{\begin{equation}}
\newcommand{\ee}{\end{equation}}
\newcommand{\beas}{\begin{eqnarray*}}
\newcommand{\eeas}{\end{eqnarray*}}
\newcommand{\bea}{\begin{eqnarray}}
\newcommand{\eea}{\end{eqnarray}}
\renewcommand{\epsilon}{\varepsilon}
\numberwithin{equation}{section}
\numberwithin{lem}{section}
\numberwithin{theo}{section}
\numberwithin{cor}{section}
\numberwithin{ex}{section}
\numberwithin{defi}{section}
\numberwithin{rem}{section}
\numberwithin{note}{section}

\newcommand{\lra}{\longrightarrow}

\begin{document}
\title[ \large O\MakeLowercase {\large n the transcendental solutions of} F{\MakeLowercase  {\large ermat type delay-differential...}}]{O\MakeLowercase  {\large n the transcendental solutions of} F{\MakeLowercase  {\large ermat type delay-differential and $c$-shift equations with some analogous results}}}
\author[A. Banerjee and T. Biswas]{Abhijit Banerjee and Tania Biswas}
\date{}
\address{ Department of Mathematics, University of Kalyani, West Bengal 741235, India.}
\email{abanerjee\_kal@yahoo.co.in, taniabiswas2394@gmail.com}
\renewcommand{\thefootnote}{}
\footnote{2010 {\emph{Mathematics Subject Classification}}: 39B32, 34M05, 30D35.}
\footnote{\emph{Key words and phrases}: Fermat type equation, delay-differential equations, shift equation, entire and meromorphic solutions, finite order, Nevanlinna theory.}
\renewcommand{\thefootnote}{\arabic{footnote}}
\setcounter{footnote}{0}
\begin{abstract} 
	In this paper, we mainly investigate on the finite order transcendental entire solutions of two Fermat types delay-differential and one Fermat type $c$-shift equations, as these types were not considered earlier. % \beas f^2(z)+R^2(z)(f^{(k)}(z+c)-f^{(k)}(z))^2=Q(z),\eeas \beas f^2(z)+R^2(z)(Af^{(m)}(z+c)+Bf^{(n)}(z))^2=1,\eeas \beas f^2(z)+L_c^2(z,f)=1,\eeas where $R(z)$, $Q(z)$ are non-zero polynomials
	Our results improve those of \cite{Liu-Dong_EJDE} in some sense. In addition, we also extend some recent results obtained in \cite{Qi-Liu-Yang-Irn}. A handful number of examples have been provided by us to justify our certain assertion as and when required.
\end{abstract}
\thanks{Typeset by \AmS -\LaTeX}
\maketitle
\section{Introduction and some basic definitions} 
 At the outset, we assume that the readers are familiar with the basic terms and notations of Nevanlinna's value distribution theory of meromorphic functions in the complex plane $\mathbb{C}$. So for such a meromorphic function $f$, terms like $T(r,f)$, $N(r, f)$, $m(r, f)$ etc., we refer to \cite{Hayman_Oxford,Laine_Gruyter}. The notation $S(r,f)$, is defined to be any quantity satisfying $S(r,f)=o(T(r,f))$ as $r\rightarrow \infty$, possibly outside a set $E$ of $r$ of finite logarithmic measure. The order of $f$ is defined by \beas \rho(f)=\limsup\limits_{r\lra \infty}\frac{\log T(r,f)}{\log r}.\eeas Moreover, the shift and difference operator of a function are represented by $f(z+c)$ and $\Delta(f) = f(z + c)-f(z),\;c \in \mathbb{C}\backslash \{0\}$, respectively. \par
In this paper, by linear $c$-shift operator in $f(z)$, we mean \bea\label{e1.1} L_{c}(z,f)=\displaystyle\sum_{j=0}^{\tau}a_jf(z+jc),\eea where $\tau\geq 1$, $a_{\tau}\neq 0$, $a_j$’s are any constants. On the other hand, for delay-differential operator in $f$, we mean, a finite sum of products of $f$, shifts of $f$, derivatives of $f$ and derivative of their shifts $f(z + jc), (c \in \mathbb{C})$, with constants coefficients.\par
We organize our paper as follows: In Section 2, we investigate on the entire
solutions of Fermat type delay-differential and $c$-shift equations and extend some previous results. The non-existence conditions of meromorphic solutions of certain non-linear $c$-shift equations will be considered in Section 3, which extend \cite{Qi-Liu-Yang-Irn}.
%It is easy to verify that, $L_c^rf(z)$ is a generalized form of $\Delta_c^pf(z)$.\par We call the linear combinations of the operators $L_{j}$$(j=1,2,3)$ as linear shift-delay-differential operator in $f(z)$ and denote it by \bea\label{e1.2}L(z,f)&=&L_1(z,f)+L_2(z,f)+L_3(z,f)\nonumber\\&=&\sum_{j=0}^{p}a_j(z)f(z+\beta_j)+\sum_{j=1}^{q}b_j(z)f^{(m_j)}(z+\delta_j)+\sum_{j=1}^{r}c_j(z)f^{(l_j)}(z)\not\equiv 0.\eea
\section{Fermat type Delay-Differential and $c$-Shift Equations}
Initially, Fermat type equations were investigated by Gross \cite{Gross-I, Gross-II},
Montel \cite{Montel}. Yang \cite{Yang_Amer} investigated the Fermat type 
equation and obtained the
following result: 
\begin{theoA}\cite{Yang_Amer}
	Let $m$, $n$ be positive integers satisfying $\frac{1}{m}+\frac{1}{n}<1$. Then there are no non-constant entire solutions $f (z )$ and $g(z )$ that satisfy \bea\label{e2.1} a(z )f ^n(z )	+ b(z )g^m(z )	= 1,\eea where $a(z )$, $b(z )$ are small functions of $f (z )$.
\end{theoA}
From {\it{Theorem A}} it is clear that either $m\geq 2$, $n>2$ or $m>2$, $n\geq 2$. So, it is natural that the case $m=n=2$ can be treated when $f(z)$ and $g(z)$ have some special relationship in (\ref{e2.1}). This was the starting point of a new era about the solution of Fermat type equations. As a result, successively several papers were published (see \cite{Chen-Gao_Korean,BQLi_Archiv,Liu_JMAA,Liu-Cao_EJDE,Liu-Cao-Cao,Liu-Yang-Al.I.Cuza,Tang-Liao,Wang-Xu-Tu_AIMS,Yang-Li,Zhang-Liao}).\par
In 2007, Tang-Liao \cite{Tang-Liao} investigated on the  transcendental meromorphic solutions of the following non-linear differential equations
\bea\label{e2.2}f(z)^2 + P(z)^2(f^{(k)}(z))^2 = Q(z),\eea  where $P(z)$, $Q(z)$ are non-zero polynomials.
 \par In 2013, Liu-Yang \cite{Liu-Yang-CMFT} considered the existence of solutions of the analogous difference equations of (\ref{e2.2}) namely \bea\label{e2.3} f (z)^2 + P(z)^2(f(z+c)-f(z))^2 = Q(z), \eea\par
In this paper, we wish to investigate on the existence of solutions of certain Fermat type delay-differential equation as follows: \bea\label{e2.4} f^2(z)+R^2(z)(f^{(k)}(z+c)-f^{(k)}(z))^2=Q(z),\eea where $R(z)$, $Q(z)$ are and non-zero polynomials.\par  Liu-Yang \cite{Liu-Yang-CMFT} proved that (\ref{e2.3}) has no finite order transcendental entire solution, that's why in (\ref{e2.4}), it will be natural to investigate the case for $k\geq 1$.
\begin{theo}\label{t1.1} If the non-linear delay-differential equation (\ref{e2.4}) has a transcendental entire solution of finite order, then $f(z)$ takes the form \beas f(z)=\frac{Q_1(z)e^{az+b}+Q_2(z)e^{-(az+b)}}{2},\eeas $a,b\in\mathbb{C}$  such that $Q_1(z)Q_2(z)=Q(z)$. Moreover, one of the following conclusions hold:
\begin{itemize}
	\item [(i)] If $e^{ac}\neq 1${\em;} then $k$ must be odd, $Q(z)$ and $R(z)$ reduce to constants satisfying the equation $2iRa^k+1=0$.
	\item [(ii)] If $e^{ac}= 1${\em;} then $\deg {R(z)}=1$, none of $Q_1(z)$, $Q_2(z)$ be constants. Also, \beas R(z)=\frac{Q_1(z)}{P(Q_1)}=\frac{Q_2(z)}{(-1)^{k-l+1}P(Q_2)},\eeas such that $P(x)=i\sum\limits_{l=0}^{k}{k\choose l}a^{k-l}[x^{(l)}(z+c)-x^{(l)}(z)].$
\end{itemize}
	 
\end{theo}
The following examples clarify that both the cases of {\it{Theorem \ref{t1.1}}} actually hold.
\begin{ex}
	The function $f(z)=\displaystyle\frac{2e^{3z+2}+3e^{-(3z+2)}}{2}$ satisfies the Fermat equation $f^2(z)-\frac{1}{36}(f'(z+c)-f'(z))^2=6$, where $c=\pi i$, $R=-\frac{1}{6i}$.
\end{ex}
\begin{ex}
	The function $f(z)=\displaystyle\frac{\alpha z e^{az+b}+\beta z e^{-(az+b)}}{2}$ satisfies the Fermat equation $f^2(z)-\frac{z^2}{a^2c^2}(f'(z+c)-f'(z))^2=\alpha\beta z^2$ such that $e^{ac}=1$, $Q_1(z)=\alpha z$, $Q_2(z)=\beta z$, where $\alpha,\beta\in\mathbb{C}\backslash\{0\}$. Here, $R=\displaystyle\frac{\alpha z}{ia(\alpha.(z+c)-\alpha. z)}=\frac{\beta z}{ia(\beta.(z+c)-\beta.z)} $, i.e., $R=\displaystyle\frac{z}{iac}$.
\end{ex}
In 2015, Liu-Dong \cite{Liu-Dong_EJDE} investigated on \bea\label{e2.5} D^2f^2(z) + (Af(z + c) + Bf(z))^2 = 1\eea and proved that if there exist finite order transcendental entire solutions of (\ref{e2.5}), then $A^2 = B^2 + D^2$. In that paper, they also discussed about \bea\label{e2.6} f^2(z) + (Af^{(m)}(z) + Bf^{(n)}(z))^2 = 1\eea that (\ref{e2.6}) admits transcendental entire solution when $m+n$ is an even and $m$, $n$ are odds.\par
In this paper, we wish to investigate on the following Fermat type equation as this type of equation was not dealt earlier: \bea\label{e2.7} f^2(z)+R^2(z)(Af^{(m)}(z+c)+Bf^{(n)}(z))^2=1,\eea where $m,n\in\mathbb{N}$, $R(z)$ be a non-zero polynomial, $A$ and $B$ are non-zero constants and prove the following theorem.
\begin{theo}\label{t1.2} If the non-linear delay-differential equation (\ref{e2.7}) has a transcendental entire solution of finite order, then $R(z)$ reduces to constant, namely, $R$ and $f(z)$ takes the form \beas f(z)=\frac{e^{az+b}+e^{-(az+b)}}{2}\eeas such that when 
	\begin{itemize}
		\item [(I)] $m$, $n$ are even, then $a^{m-n}\neq\pm\frac{B}{A}$, $R^2=\displaystyle\frac{1}{a^{2m}A^2-a^{2n}B^2}$,\\ $e^{ac}=\displaystyle\frac{-a^nB\pm\sqrt{(a^nB)^2-(a^mA)^2}}{a^mA}$. Also, $e^{ac}\not\in\left\{\pm1,-\left(\displaystyle\frac{a^mA}{a^nB}\right)^{\pm 1}\right\}$;
		\item [(II)] $m$, $n$ are odd, then $e^{ac}=\pm1$, $a^{m-n}\neq\mp\frac{B}{A}$ and $R=\displaystyle\frac{-i}{a^nB\pm a^mA}$;
		\item [(III)] $m$ is even, $n$ is odd, then $e^{ac}=\pm i$, $a^{m-n}\neq\pm i\frac{B}{A}$ and $R=\displaystyle\frac{-i}{a^nB\pm ia^mA}$;
		\item [(IV)] $n$ is even, $m$ is odd, then $a^{m-n}\neq\pm i\frac{B}{A}$, $R^2=\displaystyle-\frac{1}{a^{2n}B^2+a^{2m}A^2}$,\\ $e^{ac}=\displaystyle\frac{-a^nB\pm\sqrt{(a^nB)^2+(a^mA)^2}}{a^mA}$. Also, $e^{ac}\not\in\left\{\pm 1,\displaystyle\frac{a^mA}{a^nB},\displaystyle-\frac{a^nB}{a^mA}\right\}$.
	\end{itemize}
\end{theo}
\begin{rem}
Adopting the same procedure of {\em{Case (I)}} in {\it Theorem 2.2}, for $m=n=0$, if we choose $f(z)=Dg(z)$ and $R=1/D$ in (\ref{e2.7}), then we have $R^2=\displaystyle\frac{1}{A^2-B^2}\implies A^2=B^2+D^2$, i.e., we get \cite[Theorem 1.13]{Liu-Dong_EJDE}. So our theorem is a significant extension of the same.
\end{rem}
Following example shows that in {\it{Theorem \ref{t1.1}}}, each of the cases (I)-(IV) actually occurs.
\begin{ex}
	{\em{\textbf{(I)}}} Let $m$, $n$ both be even. Consider the function $f(z)=\frac{e^{2z+3}+e^{-(2z+3)}}{2}$. It is easy to see that $f$ satisfies the Fermat equation $f^2(z)-\frac{1}{2^7}(f''(z+c)+3f''(z))^2=1$, such that $e^{2c}=-3+2\sqrt{2}$.\\
	{\em{\textbf{(II)}}} Let $m$, $n$ be both odd. Choose  $f(z)=\frac{e^{3z+4}+e^{-(3z+4)}}{2}$, which satisfies the Fermat equation $f^2(z)-\frac{1}{{12}^2}(5f'(z+c)+f'''(z))^2=1$, such that $e^{3c}=-1$.\\
	{\em{\textbf{(III)}}} Let $m$ be even and $n$ be odd. The function $f(z)=\frac{e^{z+2}+e^{-(z+2)}}{2}$ satisfies the Fermat equation $f^2(z)-\frac{1}{8+6i}(f''(z+c)+3f'''(z))^2=1$, such that $e^{c}=i$.\\
	{\em{\textbf{(IV)}}} Let $m$ be odd, $n$ be even. Consider the function $f(z)=\frac{e^{2z+1}+e^{-(2z+1)}}{2}$ . The function $f$ satisfies the Fermat equation $f^2(z)-\frac{1}{100}(3f'(z+c)+2f''(z))^2=1$, such that $e^{2c}=\frac{1}{3}$.
\end{ex}
In \cite{Liu-Yang-CMFT}, Liu-Yang obtained the following result:
\begin{theoB}
	There is no finite order transcendental entire solution of (\ref{e2.3}), where P(z), Q(z) are non-zero polynomials.
\end{theoB}
In this paper, we partially extend the above result in the following manner.
\begin{theo}\label{t1.3} The non-linear $c$-shift equation \bea\label{e2.8} f^2(z)+L_c^2(z,f)=1\eea has finite order transcendental entire solution of the form $$f(z)=\frac{e^{az+b}+e^{-(az+b)}}{2},$$ satisfying following two equations: \bea\label{e2.9}\left\{\begin{array}{clcr}  a_0+a_1e^{ac}+a_2e^{2ac}+\ldots+a_\tau e^{\tau ac}&=&-i, \\   a_0+a_1e^{-ac}+a_2e^{-2ac}+\ldots+a_\tau e^{-\tau ac}&=&i.\end{array}\right.\eea Here $e^{ac}\neq\pm 1$. Also if $\tau=1$, $a_0\neq \pm a_1$ is required.
\end{theo}
\begin{rem} In {\em{Theorem \ref{t1.3}}}, if we take $\tau=1$ and $a_1=-a_0=1$, from (\ref{e2.9}), it is clear that when $L_c(z,f)\equiv \Delta_cf(z)$, there exists no finite order transcendental entire solution, which includes special case of {\em{Theorem B}}. 
	%However we cannot succeed to find the possible general solution of $f^2(z)+L_c^2(z,f)=Q(z)$ when $Q(z)$ is a non-zero polynomial.
\end{rem}
\begin{ex}\label{ex1.1}
	Consider the function $f(z)=\sin(\frac{\pi z}{2c})$. Here $a=\frac{i\pi}{2c}$, $e^b=\frac{1}{i}$. Then $f(z)$ be a solution of the equation $f^2(z)+(L_c(z,f))^2=1$, provided that $\tau$ is an odd integer say $\tau=2m+1,\;m\geq 1$ and the coefficients of $L_c(z,f)$  satisfy the following simultaneous equations:
	\beas\left\{\begin{array}{clcr}  a_0-a_2+a_4-a_6+\ldots+(-1)^ma_{2m}&=&0, \\  a_1-a_3+a_5-a_7+\ldots+(-1)^ma_{2m+1}&=&-1; \end{array}\right.\eeas
	and when $\tau$ is an even integer say $\tau=2m,\;m\geq 1$, the coefficients of $L_c(z,f)$  satisfy the following simultaneous equations:
	\beas\left\{\begin{array}{clcr}  a_0-a_2+a_4-a_6+\ldots+(-1)^ma_{2m}&=&0, \\  a_1-a_3+a_5-a_7+\ldots+(-1)^{m-1}a_{2m-1}&=&-1. \end{array}\right.\eeas
\end{ex}
The following lemma plays an important part for the proof in this section:
\begin{lem}\label{Borel}\cite{Yang-Yi_Kluwer}
	Suppose $f_j (z)$ $( j =1 ,2,...,n+1)$ and $g_k (z)$ $( k =1 ,2,...,n)$ $(n \geq 1)$ are entire functions satisfying the following conditions: 
	\begin{itemize}
		\item[(i)] $ \sum_{j=1}^{n} f_j (z)e^{g_j(z)} \equiv f_{n+1} (z)$, \item[(ii)] The order of $f_j (z)$ is less than the order of $e^{g_k(z)}$ for $1 \leq j \leq n +1$, $1 \leq k \leq n$ and furthermore, the order of $f_j (z)$ is less than the order of $e^{g_h(z)-g_k(z)}$ for $n \geq 2$ and $1\leq j \leq n +1$, $1\leq h<k\leq n$.
	\end{itemize} Then $f_j (z)\equiv 0$, $(j =1 ,2,...,n+1)$. 
\end{lem}
\begin{proof} [\bf\underline{Proof of Theorem \ref{t1.1}}]
	Assume that $f(z)$ is a finite order transcendental entire solution of (\ref{e2.4}), then \bea\label{e2.10}[f(z)+iR(z)(f^{(k)}(z+c)-f^{(k)}(z))][f(z)-iR(z)(f^{(k)}(z+c)-f^{(k)}(z))]=Q(z).\eea Thus both of $f(z)+iR(z)(f^{(k)}(z+c)-f^{(k)}(z))$ and $f(z)-iR(z)(f^{(k)}(z+c)-f^{(k)}(z))$ have finitely many zeros. Combining (\ref{e2.10}), with Hadamard factorization theorem, we assume that $$f(z)+iR(z)(f^{(k)}(z+c)-f^{(k)}(z))=Q_1(z)e^{P(z)}$$ and $$f(z)-iR(z)(f^{(k)}(z+c)-f^{(k)}(z))=Q_2(z)e^{-P(z)},$$ where $P(z)$ is a non-constant polynomial, otherwise $f(z)$ will be a polynomial and $Q_1(z)Q_2(z)=Q(z)$, where $Q_1(z)$, $Q_2(z)$ are non-zero polynomials. We denote the degrees of $Q_1(z)$, $Q_2(z)$ and $P(z)$ by $k_1$, $k_2$ and $k_3$, respectively. Thus, we have, \bea\label{e2.11}f(z)=\frac{Q_1(z)e^{P(z)}+Q_2(z)e^{-P(z)}}{2}\eea and \bea\label{e2.12}f^{(k)}(z+c)-f^{(k)}(z)=\frac{Q_1(z)e^{P(z)}-Q_2(z)e^{-P(z)}}{2iR(z)}.\eea From (\ref{e2.11}), we have 
	\bea\label{e2.13} f^{(k)}(z)=\frac{p_1(z)e^{P(z)}+p_2(z)e^{-P(z)}}{2},\eea where, 
	\bea\label{e2.14} p_1(z)&=&Q_1(z)\left[P'(z)^k+M_{1,k}\left(P',P'',\ldots,P^{(k)}\right)\right]+Q_1'(z)M_{2,k-1}\left(P',P'',\ldots,P^{(k-1)}\right)\nonumber\\&&+\ldots+Q_1^{(k-1)}(z)M_{k,1}(P')+Q_1^{(k)}(z),\eea \bea\label{e2.15} p_2(z)&=&Q_2(z)\left[(-1)^kP'(z)^k+N_{1,k}\left(P',P'',\ldots,P^{(k)}\right)\right]+(-1)^{k-1}Q_2'(z)N_{2,k-1}\left(P',P'',\right.\nonumber\\&&\left.\ldots,P^{(k-1)}\right)+\ldots-Q_2^{(k-1)}(z)N_{k,1}(P')+Q_2^{(k)}(z),\eea
	where $M_{j,k-j+1}(N_{j,k-j+1})$ $(j=1,2)$ are differential polynomials of $P'$ with degree $k-1$. $M_{j,k-j+1}(N_{j,k-j+1})$ are differential polynomial of $P'$ with degree $k-j+1$ $(j=3,4,\ldots,k)$. It follows that $p_1(z)$, $p_2(z)$ are polynomials with
	degree $k_1 +k(k_3 -1)\geq k_1$ and  $k_2 +k(k_3 -1)\geq k_2$, respectively. Using (\ref{e2.12}) and (\ref{e2.13}) we have
	\bea\label{e2.16} &&\frac{Q_1(z)e^{P(z)}-Q_2(z)e^{-P(z)}}{2iR(z)}\nonumber\\&=&\frac{p_1(z+c)e^{P(z+c)}+p_2(z+c)e^{-P(z+c)}}{2}-\frac{p_1(z)e^{P(z)}+p_2(z)e^{-P(z)}}{2}\nonumber\\&=&\frac{(p_1(z+c)e^{\Delta_c{P(z)}}-p_1(z))e^{P(z)}+(p_2(z+c)e^{-\Delta_c{P(z)}}-p_2(z))e^{-P(z)}}{2}.\eea
	Then (\ref{e2.16}) can be written as 
	\bea\label{e2.17} &&\left[p_1(z+c)e^{\Delta_c{P(z)}}-p_1(z)-\frac{Q_1(z)}{iR(z)}\right]e^{P(z)}\nonumber\\&&+\left[p_2(z+c)e^{-\Delta_c{P(z)}}-p_2(z)+\frac{Q_2(z)}{iR(z)}\right]e^{-P(z)}=0.\eea
	Applying {\it Lemma \ref{Borel}} on (\ref{e2.17}), we have 
	\bea\label{e2.18} p_1(z+c)e^{\Delta_c{P(z)}}-p_1(z)-\frac{Q_1(z)}{iR(z)}=0\eea and \bea\label{e2.19} p_2(z+c)e^{-\Delta_c{P(z)}}-p_2(z)+\frac{Q_2(z)}{iR(z)}=0.\eea
	Now we show that $P(z)$ is a one-degree polynomial. If not, suppose that $\deg(P(z))\geq 2$. Then applying {\it Lemma \ref{Borel}} on (\ref{e2.18}) and (\ref{e2.19}), we have $p_1(z+c)=0$ and $p_2(z+c)=0$, a contradiction, which implies that $P(z)$ is a one-degree polynomial, say, 
	\bea\label{e2.20} P(z)=az+b,\;a,b\in\mathbb{C}.\eea
	In view of (\ref{e2.20}), (\ref{e2.17}) yields 
	\bea\label{e2.21} \left[p_1(z+c)e^{ac}-p_1(z)-\frac{Q_1(z)}{iR(z)}\right]e^{P(z)}+\left[p_2(z+c)e^{-ac}-p_2(z)+\frac{Q_2(z)}{iR(z)}\right]e^{-P(z)}=0.\eea
	Using (\ref{e2.20}), the expressions of $p_1(z)$ and $p_2(z)$ given by (\ref{e2.14}) and (\ref{e2.15}) reduec to 
	\bea\label{e2.22} p_1(z)=\sum\limits_{l=0}^{k}{k\choose l}a^{k-l}Q_1^{(l)}(z)\;\;and\;\;p_2(z)=\sum\limits_{l=0}^{k}{k\choose l}(-a)^{k-l}Q_2^{(l)}(z),\eea respectively.
	Now we have to consider the following two cases:\\
	{\bf Case 1:} If $e^{ac}\neq 1$. Then considering the degrees of $p_1$, $Q_1$ and $p_2$, $Q_2$ of the equation (\ref{e2.21}), one can conclude that $R(z)$ is constant, say, $R$.
	Thus, using (\ref{e2.20}) and (\ref{e2.22}), (\ref{e2.18}) and (\ref{e2.19}) become 
	\bea\label{e2.23} iR\sum\limits_{l=0}^{k}{k\choose l}a^{k-l}\left[e^{ac}Q_1^{(l)}(z+c)-Q_1^{(l)}(z)\right]=Q_1(z)\eea and 
	\bea\label{e2.24} iR\sum\limits_{l=0}^{k}{k\choose l}(-a)^{k-l}\left[Q_2^{(l)}(z)-e^{-ac}Q_2^{(l)}(z+c)\right]=Q_2(z).\eea Considering the highest degree on both sides of (\ref{e2.23}) and (\ref{e2.24}), we have 
	\beas iRa^k(e^{ac}-1)=1\; and\; iR(-a)^k(1-e^{-ac})=1,\eeas which implies $e^{ac}=(-1)^k$. Since $e^{ac}\neq 1$, $k$ must be odd and \bea\label{e2.25}a=\frac{(2m+1)\pi i}{c},\eea $m$ is any integer. So, 
	\bea\label{e2.26}2iRa^k+1=0.\eea
	Eliminating $e^{\Delta_cP(z)}$ from (\ref{e2.18}) and (\ref{e2.19}), we have 
	\bea\label{e2.27} && \frac{iRp_1(z)+Q_1(z)}{iRp_1(z+c)}=\frac{iRp_2(z+c)}{iRp_2(z)-Q_2(z)}\nonumber\\&\implies& (p_1(z)p_2(z)-p_1(z+c)p_2(z+c))R^2+(p_1(z)Q_2(z)-p_2(z)Q_1(z))iR\nonumber\\&&+Q_1(z)Q_2(z)=0.\eea
	\par First suppose $Q(z)$ be constant, then $Q_1(z)$ and $Q_2(z)$ are constants, then from (\ref{e2.22}) and (\ref{e2.27}) we have, 
	$(2iRa^k+1)Q_1Q_2=0$, which is possible from (\ref{e2.26}).\par Suppose that $Q(z)$ be one-degree polynomial and let $Q_1(z)=\alpha_1z+\alpha_0$, $\alpha_1\neq 0$ and $Q_2(z)$ be constant. Using $e^{ac}=-1$, (\ref{e2.25}), (\ref{e2.26}), we have from (\ref{e2.23}) that \beas &&iR[a^{k}\{-(\alpha_1(z+c)+\alpha_0)-(\alpha_1z+\alpha_0)\}-2ka^{k-1}\alpha_1]=\alpha_1z+\alpha_0\\&\implies&iR[a^{k}\{2(\alpha_1z+\alpha_0)+\alpha_1c\}+2ka^{k-1}\alpha_1]+(\alpha_1z+\alpha_0)=0\\&\implies&(2iRa^k+1)(\alpha_1z+\alpha_0)+iRa^{k-1}\alpha_1(ac+2k)=0\\&\implies&2k=-(2m+1)\pi i,\eeas a contradiction. Similarly, considering $Q_1(z)$ and $Q_2(z)$ respectively as constant and one-degree polynomial, we can get a contradiction in a similar way. So, $Q(z)$ cannot be one-degree polynomial.\par Next suppose that $Q(z)$ is a polynomial of degree $2$. If $Q_1(z)$ and $Q_2(z)$ both are one-degree polynomials, then from previous argument, we get a contradiction. Now let $Q_1(z)=\beta_2z^2+\beta_1z+\beta_0$, $\beta_2\neq 0$ and $Q_2(z)$ be a constant. Using $e^{ac}=-1$, (\ref{e2.25}), (\ref{e2.26}), we have from (\ref{e2.23}) that \beas &&iR[a^{k}\{-(\beta_2(z+c)^2+\beta_1(z+c)+\beta_0)-(\beta_2z^2+\beta_1z+\beta_0)\}\\&&+ka^{k-1}\{-(2\beta_2(z+c)+\beta_1)-(2\beta_2z+\beta_1)\}-2k(k-1)a^{k-2}\beta_2]\\&&=\beta_2z^2+\beta_1z+\beta_0\\&\implies&(2iRa^k+1)(\beta_2z^2+\beta_1z+\beta_0)+2iRa^{k-1}\beta_2(ac+2k)z%\\&&+a^k(\alpha_2c+\alpha_1)c+ka^{k-1}(2\alpha_2c+3\alpha_1)+2k(k-1)a^{k-2}\alpha_2
	+P_0(z)=0,\eeas where $P_0(z)$ is a polynomial of degree $0$. Now, comparing the coefficient of $z$ from both side of the above equation, again we have $2k=-(2m+1)\pi i$, a contradiction. Similarly, considering $Q_1(z)$ and $Q_2(z)$ respectively as constant and two-degree polynomial, again we get a contradiction. Thus $Q(z)$ cannot be a second degree polynomial. \par 
	% Now suppose that $Q_1(z)=\gamma_3z^3+\gamma_2z^2+\gamma_1z+\gamma_0$, $\gamma_3\neq 0$. Again using $e^{ac}=-1$, (\ref{4.13a}), (\ref{4.14}), we have from (\ref{4.12}) that \beas (2iRa^k+1)(\gamma_3z^3+\gamma_2z^2+\gamma_1z+\gamma_0)-3iRa^{k-1}\gamma_3(ac+2k)z^2+P_1(z)=0,\eeas $P_1(z)$ is a polynomial of degree $1$. Here, comparing the coefficient of $z^2$ from both side of the above equation, again we have $2k=-(2m+1)\pi i$, a contradiction. Similarly, considering $Q_2(z)$ as two degree or one degree polynomial, we get a contradiction. Repeating the above process, we can say that, if $Q_(z)$ is a polynomial of degree $n$, then whenever $Q_i(z)$ $i=1$, $2$ is a polynomial of degree $n-j$, $j=0,1,2,\ldots,n$ we can deduce a contradiction.
	Now suppose that $Q(z)$ is a polynomial of degree $n\geq3$. Let $Q_1(z)=\gamma_{j}z^{j}+\gamma_{j-1}z^{j-1}+\ldots+\gamma_0$, $\gamma_{j}\neq 0$ and $Q_2(z)$ be of degree $n-j$, $3\leq j\leq n$ . Again using $e^{ac}=-1$, (\ref{e2.25}), (\ref{e2.26}), we have from (\ref{e2.23}) that \beas (2iRa^k+1)Q_1(z)+niRa^{k-1}\gamma_n(ac+2k)z^{j-1}+P_{j-2}(z)=0,\eeas $P_{j-2}(z)$ is a polynomial of degree $j-2$. Here, comparing the coefficient of $z^{j-1}$ from both sides of the above equation, again we have $2k=-(2m+1)\pi i$, a contradiction. Similarly, $Q_2(z)$ cannot be a polynomial of degree $n-j$, as in that case also we get a contradiction. So, $Q(z)$ cannot be a polynomial of degree $n\geq 3$.\par Thus $Q(z)$ must be constant, say $Q$. Therefore, we must have the form of the solution is \beas f(z)=\frac{Q_1e^{az+b}+Q_2e^{-az-b}}{2},\eeas such that $Q_1Q_2=Q$.\\
	{\bf Case 2:} If $e^{ac}=1$, then (\ref{e2.21}) becomes \bea\label{e2.28} \left[p_1(z+c)-p_1(z)-\frac{Q_1(z)}{iR(z)}\right]e^{P(z)}+\left[p_2(z+c)-p_2(z)+\frac{Q_2(z)}{iR(z)}\right]e^{-P(z)}=0.\eea Applying {\it Lemma \ref{Borel}} and using (\ref{e2.22}) on (\ref{e2.28}), we have \bea\label{e2.29} iR(z)\sum\limits_{l=0}^{k}{k\choose l}a^{k-l}\left[Q_1^{(l)}(z+c)-Q_1^{(l)}(z)\right]=Q_1(z)\eea and 
	\bea\label{e2.30} iR(z)\sum\limits_{l=0}^{k}{k\choose l}(-a)^{k-l}\left[Q_2^{(l)}(z)-Q_2^{(l)}(z+c)\right]=Q_2(z).\eea
	Note that the highest degree of $\sum\limits_{l=0}^{k}{k\choose l}a^{k-l}[Q_1^{(l)}(z+c)-Q_1^{(l)}(z)]$ and $\sum\limits_{l=0}^{k}{k\choose l}(-a)^{k-l}[Q_2^{(l)}(z)-Q_2^{(l)}(z+c)]$ are $k_1-1$ and $k_2-1$ respectively. Also we see that none of $Q_1(z)$ and $Q_2(z)$ be constants, i.e., $\deg{Q(z)}\geq 2$. Comparing the total degree of the equations (\ref{e2.29}) and (\ref{e2.30}), we can deduce that $\deg {R(z)}=1$. Moreover, \beas R(z)=\frac{Q_1(z)}{P(Q_1)}=\frac{Q_2(z)}{(-1)^{k-l+1}P(Q_2)}\eeas such that $$P(x)=i\sum\limits_{l=0}^{k}{k\choose l}a^{k-l}\left[x^{(l)}(z+c)-x^{(l)}(z)\right].$$
\end{proof}

\begin{proof} [\bf\underline{Proof of Theorem \ref{t1.2}}]
	Assume that $f(z)$ is a finite order transcendental entire solution of (\ref{e2.7}), then \bea\label{e2.31}[f(z)+iR(z)(Af^{(m)}(z+c)+Bf^{(n)}(z))][f(z)-iR(z)(Af^{(m)}(z+c)+Bf^{(n)}(z))]=1.\eea Thus both of $f(z)+iR(z)(Af^{(m)}(z+c)+Bf^{(n)}(z))$ and $f(z)-iR(z)(Af^{(m)}(z+c)+Bf^{(n)}(z))$ have no zeros. Combining (\ref{e2.31}), with Hadamard factorization theorem, we assume that $$f(z)+iR(z)(Af^{(m)}(z+c)+Bf^{(n)}(z))=e^{P(z)}$$ and $$f(z)-iR(z)(Af^{(m)}(z+c)+Bf^{(n)}(z))=e^{-P(z)},$$ where $P(z)$ is a non-constant polynomial, otherwise $f(z)$ will be constant. Thus we have, \bea\label{e2.32}f(z)=\frac{e^{P(z)}+e^{-P(z)}}{2}\eea and \bea\label{e2.33}Af^{(m)}(z+c)+Bf^{(n)}(z)=\frac{e^{P(z)}-e^{-P(z)}}{2iR(z)}.\eea Using (\ref{e2.32}) in (\ref{e2.33}), we have 
	\bea\label{e2.34} &&\frac{e^{P(z)}-e^{-P(z)}}{2iR(z)}\nonumber\\&&=\frac{A}{2}\left[p_1(z+c)e^{P(z+c)}+p_2(z+c)e^{-P(z+c)}\right]+\frac{B}{2}\left[q_1(z)e^{P(z)}+q_2(z)e^{-P(z)}\right],\eea where 
	\beas p_1(z+c)=P'(z+c)^m+M_{1,m-1}\left(P'(z+c),P''(z+c),\ldots,P^{(m)}(z+c)\right),\eeas
	\beas p_2(z+c)=(-1)^mP'(z+c)^m+M_{2,m-1}\left(P'(z+c),P''(z+c),\ldots,P^{(m)}(z+c)\right),\eeas
	\beas q_1(z)=P'(z)^n+N_{1,n-1}\left(P',P'',\ldots,P^{(n)}\right),\eeas \beas q_2(z)=(-1)^nP'(z)^n+N_{2,n-1}\left(P',P'',\ldots,P^{(n)}\right),\eeas 
	where $M_{i,m-1}$, $N_{i,n-1}$ $(i=1,2)$ are differential polynomials of $P'$ with degree $m-1$, $n-1$.
	Then (\ref{e2.34}) can be written as 
	\bea\label{e2.35} &&\left[iAR(z)p_1(z+c)e^{\Delta_c{P(z)}}+iBR(z)q_1(z)-1\right]e^{P(z)}\nonumber\\&&+\left[iAR(z)p_2(z+c)e^{-\Delta_c{P(z)}}+iBR(z)q_2(z)+1\right]e^{-P(z)}=0.\eea
	Applying {\it Lemma \ref{Borel}} on (\ref{e2.35}), we have 
	\bea\label{e2.36} iAR(z)p_1(z+c)e^{\Delta_c{P(z)}}+iBR(z)q_1(z)-1=0\eea and 
	\bea\label{e2.37} iAR(z)p_2(z+c)e^{-\Delta_c{P(z)}}+iBR(z)q_2(z)+1=0.\eea
	Now we show that $P(z)$ is a one-degree polynomial. If not, suppose that $\deg(P(z))\geq 2$. Applying {\it Lemma \ref{Borel}} on (\ref{e2.36}) and (\ref{e2.37}), we have $p_1(z+c)=0$ and $p_2(z+c)=0$, a contradiction, which concludes that $P(z)$ is a one-degree polynomial, say, 
	\bea\label{e2.38} P(z)=az+b.\eea Using (\ref{e2.38}), we have 
	\bea\label{e2.39} p_1(z)=a^m,\;p_2(z)=(-a)^m,\;\;q_1(z)=a^n,\;\;q_2(z)=(-a)^n.\eea
	Using (\ref{e2.38}) and (\ref{e2.39}) in (\ref{e2.35}), we have 
	\bea\label{e2.40} &&\left[iAa^mR(z)e^{ac}+iBa^nR(z)-1\right]e^{az+b}\nonumber\\&&+\left[iA(-a)^mR(z)e^{-ac}+iB(-a)^nR(z)+1\right]e^{-az-b}=0.\eea Again applying {\it Lemma \ref{Borel}} on (\ref{e2.40}), we have \bea\label{e2.41} iAa^mR(z)e^{ac}+iBa^nR(z)-1=0\;and\;iA(-a)^mR(z)e^{-ac}+iB(-a)^nR(z)+1=0.\eea From here, we can conclude that $R(z)$ is constant, say, $R$. So, (\ref{e2.41}) becomes 
	\bea\label{e2.42} iR(a^me^{ac}A+a^nB)=1\eea and \bea\label{e2.43}iR((-a)^me^{-ac}A+(-a)^nB)=-1.\eea
	
	\textbf{Case I:} Let $m$, $n$ be even. Then (\ref{e2.43}) becomes \bea\label{e2.44}iR(a^me^{-ac}A+a^nB)=-1.\eea Now eliminating $e^{ac}$ from (\ref{e2.42}) and (\ref{e2.44}) we get, $R^2(a^{2m}A^2-a^{2n}B^2)=1$, which implies, $\displaystyle\frac{a^mA}{a^nB}\neq\pm1$, i.e., $a^{m-n}\neq\pm\frac{B}{A}$.\\ Again from (\ref{e2.42}) and (\ref{e2.44}) we get, $e^{ac}=\displaystyle\frac{-a^nB\pm\sqrt{(a^nB)^2-(a^mA)^2}}{a^mA}$.\\ Also, $e^{ac}\not\in\left\{\pm1,\left(\displaystyle\frac{a^mA}{a^nB}\right)^{\pm 1}\right\}$.
	
	\textbf{Case II:} Let $m$, $n$ be odd. Then (\ref{e2.43}) becomes \bea\label{e2.45}iR(a^me^{-ac}A+a^nB)=1.\eea So, in this case,  from (\ref{e2.42}) and (\ref{e2.45}) we get, $e^{ac}=\pm1$. Also, $a^{m-n}\neq\mp\frac{B}{A}$. Then $R=\displaystyle\frac{-i}{a^nB\pm a^mA}$.\\
	
	\textbf{Case III:} Let $m$ be even and $n$ be odd. Then (\ref{e2.43}) becomes \bea\label{e2.46}iR(a^me^{-ac}A-a^nB)=-1.\eea So, in this case, from (\ref{e2.42}) and (\ref{e2.46}) we get, $e^{ac}=\pm i$ and $a^{m-n}\neq\pm i\frac{B}{A}$. Then $R=\displaystyle\frac{-i}{a^nB\pm ia^mA}$.\\
	
	\textbf{Case IV:} Let $m$ be odd and $n$ be even. Then (\ref{e2.43}) becomes \bea\label{e2.47}iR(a^me^{-ac}A-a^nB)=1.\eea Now eliminating $e^{ac}$ from (\ref{e2.42}) and (\ref{e2.47}) we get, $R^2(a^{2n}B^2+a^{2m}A^2)=-1$, which implies, $\displaystyle\frac{a^mA}{a^nB}\neq\pm i$. Again from (\ref{e2.42}) and (\ref{e2.47}) we get, $e^{ac}=\displaystyle\frac{-a^nB\pm\sqrt{(a^nB)^2+(a^mA)^2}}{a^mA}$. Also, $e^{ac}\not\in\left\{\pm 1,\displaystyle\frac{a^mA}{a^nB},\displaystyle-\frac{a^nB}{a^mA}\right\}$.
	
\end{proof}

\begin{proof} [\bf\underline{Proof of Theorem \ref{t1.3}}]
	Assume that $f(z)$ is a finite order transcendental entire solution of (\ref{e2.8}), then \bea\label{e2.48}[f(z)+iL_c(z,f)][f(z)-iL_c(z,f)]=1.\eea Proceeding in the same way as done in the previous theorem, from (\ref{e2.48}), we have, \bea\label{e2.49}f(z)=\frac{e^{P(z)}+e^{-P(z)}}{2}\eea and \bea\label{e2.50}L_c(z,f)=\frac{e^{P(z)}-e^{-P(z)}}{2i}.\eea From (\ref{e1.1}), (\ref{e2.49}) and (\ref{e2.50}), we have 
	\bea\label{e2.51} &&\frac{e^{P(z)}-e^{-P(z)}}{2i}=\sum\limits_{j=0}^{\tau}a_j\frac{e^{P(z+jc)}+e^{-P(z+jc)}}{2}\nonumber\\&\implies& \sum\limits_{j=1}^{\tau}a_j\left(e^{P(z+jc)}+e^{-P(z+jc)}\right)=-(a_0+i)e^{P(z)}-(a_0-i)e^{-P(z)}.\eea
	Then (\ref{e2.51}) can be written as 
	\bea\label{e2.52} &&\left(a_0+i+\sum\limits_{j=1}^{\tau}a_je^{\Delta_{jc}P(z)}\right)e^{P(z)}+\left(a_0-i+\sum\limits_{j=1}^{\tau}a_je^{-\Delta_{jc}{P(z)}}\right)e^{-P(z)}=0.\eea
	Applying {\it Lemma \ref{Borel}} on (\ref{e2.52}), we have 
	\bea\label{e2.53} a_0+i+\sum\limits_{j=1}^{\tau}a_je^{\Delta_{jc}P(z)}=0\eea and 
	\bea\label{e2.54} a_0-i+\sum\limits_{j=1}^{\tau}a_je^{-\Delta_{jc}{P(z)}}=0.\eea
	Now we show that $P(z)$ is a one-degree polynomial. On the contrary, suppose that $\deg(P(z))\geq 2$. Applying {\it Lemma \ref{Borel}} on (\ref{e2.53}) and (\ref{e2.54}), we have $a_j=0$ for all $1\leq j\leq \tau$, which in view of (\ref{e2.8}) implies that $f(z)$ is constant, a contradiction. So, $P(z)$ is a one-degree polynomial, say, $P(z)=az+b$. Then using (\ref{e2.53}) and (\ref{e2.54}), the relation between $a$, $c$ and $a_j$, $0\leq j\leq\tau$, can be determined by (\ref{e2.9}). Also from (\ref{e2.9}), it is clear that $e^{ac}\neq\pm1$. If $\tau=1$, $a_0\neq \pm a_1$ is required.
\end{proof}
\section{Non-linear $c$-Shift Equations}
For the existence of solutions of non-linear $c$-shift equation, in 2011, Qi \cite{Xi-Qi-Polon} obtained the following theorems:
\begin{theoC}\cite{Xi-Qi-Polon}
	Let $q(z)$, $p(z)$ be polynomials and let $n$, $m$ be distinct positive integers. Then the equation \bea\label{e3.1} f^m(z) + q(z)f(z + c)^n = p(z)\eea has no transcendental entire solutions of finite order. 
\end{theoC}
In 2015, Qi-Liu-Yang \cite{Qi-Liu-Yang-Irn} obtained the meromorphic variant of {\it{Theorem C}} and improved this as follows: 
\begin{theoD}\cite{Qi-Liu-Yang-Irn}
	Let $f(z)$ be a transcendental meromorphic function with finite order, $m$ and $n$ be two positive integers such that $m \geq n + 4$, $p(z)$ be a meromorphic function satisfying $\ol N\left(r,\frac{1}{p(z)}\right) = S(r,f)$ and $q(z)$ be a non-zero meromorphic function satisfying that $T(r,q(z)) = S(r,f)$. Then, $f(z)$ is not a solution of equation \bea\label{e3.2} f^m(z) + q(z)f(z + c)^n = p(z).\eea
\end{theoD}
\begin{theoE}\cite{Qi-Liu-Yang-Irn}
	Let $f(z)$ be a transcendental entire function with finite order, $m$ and $n$ be two positive integers such that $m \geq n + 2$, $p(z)$ be a meromorphic function satisfying $\ol N\left(r,\frac{1}{p(z)}\right) = S(r,f)$ and $q(z)$ be a non-zero meromorphic function satisfying that $T(r,q(z)) = S(r,f)$. Then $f(z)$ is not a solution of equation (\ref{e3.2}).
\end{theoE}
In this paper we extend {\it{Theorems D-E}} at the expense of replacing $f(z+c)$ by $L_c(z,f)$.
\begin{theo}\label{t2.4}
	Let $f(z)$ be a transcendental meromorphic function with finite order, $ m$ and $n$ be two positive integers such that $m\geq (\tau+1)(n+2) + 2$, $p(z)$ be a meromorphic function satisfying $\ol N\left(r,\frac{1}{p(z)}\right) = S(r,f)$ and $q(z)$ be a non-zero meromorphic function satisfying that $T(r,q(z)) = S(r,f)$. Then, $f(z)$ is not a solution of the non-linear $c$-shift equation  \bea\label{e3.3}f^m(z) + q(z)(L_c(z,f))^n = p(z).\eea
\end{theo}
\begin{cor}\label{c1.1}
	\label{t5}Let $f(z)$ be a transcendental entire function with finite order, $ m$ and $n$ be two positive integers such that $m\geq n+2$, $p(z)$ be a meromorphic function satisfying $\ol N\left(r,\frac{1}{p(z)}\right) = S(r,f)$ and $q(z)$ be a non-zero meromorphic function satisfying that $T(r,q(z)) = S(r,f)$. Then, $f(z)$ is not a solution of the non-linear $c$-shift equation (\ref{e3.3}).
\end{cor}
The next examples show that if the condition $m\geq n+2$ is omitted then the equation (\ref{e3.3}) can admit a transcendental entire solution.\par First considering $n=1$ and $m=2$ we have the following examples.
\begin{ex}
	For an odd integer $s$, the function $f(z)=e^\frac{s\pi iz}{c}+z$ is a solution of the equation $f^2(z)-zL_c(z,f)=e^\frac{2s\pi iz}{c}$, for $k\geq 2$, provided that the coefficients of $L_c(z,f)$  satisfy the following simultaneous equations:
	\beas  \left\{\begin{array}{clcr} & a_0-a_1+a_2-a_3+\ldots+(-1)^ka_k&=2, \\   & a_0+a_1+a_2+a_3+a_4+\ldots+a_k&=1, \\   & a_1+2a_2+3a_3+4a_4+\ldots+ka_k&=0.\end{array}\right.\eeas
\end{ex}
Next considering $m=n=1$ we have the following example.
\begin{ex}
	The function $f(z)=ze^{\frac{\pi iz}{c}}$ satisfies the equation $f(z)+\frac{1}{z+1}L_c(z,f)=\frac{z(z+2)}{z+1}e^{\frac{\pi iz}{c}}$ where the coefficients of  $L_c(z,f)$ is chosen such that they satisfy simultaneously the equations 
	\beas  \left\{\begin{array}{clcr}  a_0-a_1+a_2-\ldots +(-1)^{k}a_k &=&1, \\   -a_1+2a_2-3a_3+\ldots +k(-1)^{k}a_k&=&0.\end{array}\right.\eeas	
\end{ex}
To proceed further we require the following lemmas:
%\begin{lem}\label{Clunie} (\cite{Halburd-Korhonen}, Corollary 3.3)Let $f$ be a transcendental meromorphic solution with finite order of \beas f^n(z)P(z,f)=Q(z,f),\eeas	where $P(z,f),\;Q(z,f)$ are difference polynomials in $f$ and its shifts with small meromorphic coefficients $a_\lambda,\;\lambda\in I$. If the total degree of $Q(z,f)$ as a polynomial in $f$ and its shifts are $\leq n$, then	\beas m(r,P(z,f)) = S(r,f)\eeas	for all $r$ outside of a possible exceptional set with finite logarithmic measure.\end{lem}
\begin{lem}\label{l6} \cite[Lemma 5.1]{Chiang-Feng}
	Let $f(z)$ be a finite order meromorphic function and $\epsilon>0$, then $T(r,f(z+c))=T(r,f(z))+o(r^{\sigma-1+\epsilon})+O(\log r)$ and $\sigma(f(z+c))=\sigma(f(z))$. Thus, if $f(z)$ is a transcendental meromorphic function with finite order, then we know $T(r,f(z+c))=T(r,f)+S(r,f)$.
\end{lem}
\begin{lem}\label{l7} \cite[Theorem 2.1]{Halburd-Korhonen-Fenn}
	Let $f(z)$ be a meromorphic function with finite order, and let $c \in \mathbb{C}$ and $\delta\in (0,1)$. Then $m\left(r, \frac{f(z + c)}{ f(z)} \right)+ m\left(r, \frac{f(z)}{ f(z + c)}\right)= o\left(\frac{T(r,f)}{ r^\delta} \right)= S(r,f)$. 
\end{lem}
\begin{lem}\label{l8}\cite{Heittokangas+4} 
	Let $f$ be a non-constant meromorphic function of finite order and $c\in\mathbb{C}$. Then 
	\beas N(r,\infty;f(z+c))\leq N(r,\infty;f(z))+S(r,f),\;\;\;\ol N(r,\infty;f(z+c))\leq \ol N(r,\infty;f)+S(r,f).\eeas 
\end{lem}
\begin{proof}[\bf\underline{Proof of Theorem \ref{t2.4}}]
	Suppose by contradiction that $f(z)$ is a transcendental meromorphic function with finite order satisfying equation (\ref{e3.3}).\par If $T(r,p(z)) = S(r,f)$, then applying {\it Lemma \ref{l6}} to equation (\ref{e3.3}), we have \beas m.T(r,f)&=&T(r,f^m)\\&=&T(r,p(z)-q(z)(L_c(z,f))^n)\\&=&T(r,L_c(z,f)^n)+S(r,f)\\&\leq& (\tau+1)n.T(r,f)+S(r,f),\eeas
	which contradicts the assumption that $m \geq (\tau+1)(n+2) + 2$.\par If $T(r,p(z)) \neq S(r,f)$, differentiating equation (\ref{e3.3}), we get \bea\label{e3.4}(f^m(z))^{\prime} + (q(z)(L_c(z,f))^n)^{\prime} = p^{\prime}(z).\eea
	\par Next dividing (\ref{e3.4}) by (\ref{e3.3}) we have 
	\bea\label{e3.5}&& p^{\prime}(z)[f^m(z)+q(z)(L_c(z,f))^n]=p(z)[(f^m(z))^{\prime}+(q(z)(L_c(z,f))^n)^{\prime}]\nonumber\\ &\implies& f^m(z)=\displaystyle\frac{\frac{p^{\prime}(z)}{p(z)}q(z)(L_c(z,f))^n-(q(z)(L_c(z,f))^n)^{\prime}}{\frac{(f^m(z))^{\prime}}{f^m(z)}-\frac{p^{\prime}(z)}{p(z)}}. \eea
	\par First observe that $\frac{(f^m(z))^{\prime}}{ f^m(z)} -\frac{p^{\prime}(z)}{p(z)}$  cannot vanish identically. Indeed, if $\frac{(f^m(z))^{\prime}}{f^m(z)}-\frac{p^{\prime}(z)}{p(z)}\equiv 0$, then we get $p(z) = \alpha f^m(z)$, where $\alpha$ is a non-zero constant. Substituting the above equality to equation (\ref{e3.3}), we have $q(z)(L_c(z,f))^n = (\alpha-1)f^m(z)$. From {\it Lemma \ref{l6}} and the above equation, we immediately see as above that $mT(r,f) \leq (\tau+1) nT(r,f) + S(r,f)$, which is a contradiction to $m \geq (\tau+1)(n+2) + 2$.  
	From equation (\ref{e3.5}), we know \bea\label{e3.6} mT(r,f) &=& T(r,f^m)\nonumber\\ &\leq& m\left(r,q(z)(L_c(z,f))^n\right) + m\left(r,\frac{p^{\prime}(z)}{p(z)} - \frac{(q(z)(L_c(z,f))^n)^{\prime}}{ q(z)(L_c(z,f))^n }\right)\nonumber\\& +& N\left(r, \frac{p^{\prime}(z)}{p(z)}\; q(z)(L_c(z,f))^n-(q(z)(L_c(z,f))^n)^{\prime}\right)\nonumber\\& +& m\left(r,\frac{ (f^m(z))^{\prime}}{ f^m(z)} - \frac{p^{\prime}(z)}{p(z)}\right)+ N\left(r, \frac{(f^m(z))^{\prime}}{ f^m(z)}-{ \frac{p^{\prime}(z)}{p(z)}}\right)+ S(r,f).\eea 
	As {\it Lemma \ref{l6}} together with equation (\ref{e3.3}) implies that \beas(m-(\tau+1)n)T(r,f) + S(r,f) \leq T(r,p(z)) \leq (m + (\tau+1)n)T(r,f) + S(r,f),\eeas 
	we conclude that \bea\label{e3.7} S(r,p(z)) = S(r,f).\eea Applying {\it Lemmas \ref{l6}}, {\it \ref{l7}} and (\ref{e3.7}) to equation (\ref{e3.6}), we obtain that \bea\label{e3.8} mT(r,f) &\leq& nm(r,f) + N\left(r, \frac{p^{\prime}(z)}{p(z)} q(z)(L_c(z,f))^n-(q(z)(L_c(z,f))^n)^{\prime}\right) \nonumber\\&& +N\left(r, \frac{(f^m(z))^{\prime}}{ f^m(z)}-{ \frac{p^{\prime}(z)}{p(z)}}\right)+ S(r,f).\eea
	Let \bea\label{e3.9} H(z)= \frac{p^{\prime}(z)}{p(z)} q(z)(L_c(z,f))^n-(q(z)(L_c(z,f))^n)^{\prime}\eea and 
	\bea\label{e3.10} G(z)=\frac{(f^m(z))^{\prime}}{ f^m(z)}-{ \frac{p^{\prime}(z)}{p(z)}}. \eea
	\par First of all, we deal with $N(r,H(z))$. From (\ref{e3.3}) and (\ref{e3.9}), we know the poles of $H(z)$ are at the zeros of $p(z)$ and at the poles of $f(z)$, $f(z +jc)$, $(j=1,2,\ldots,\tau)$ and $q(z)$. Poles of $p(z)$ will not contribute towards the poles of $H(z)$ as from the equation (\ref{e3.3}) we know that the poles of $p(z)$ should be at the poles of $f(z)$, $f(z + jc)$, $(j=1,2,\ldots,\tau)$ and $q(z)$. We note that $T(r,q(z)) = S(r,f)$. \par 
	If $z_0$ is a zero of $p(z)$ then by (\ref{e3.9}), $z_0$ is at most a simple pole of $H(z)$. If $z_0$ is a pole of $f(z)$ of multiplicity $t$ but not a pole of $f(z + jc)$, $j=1,\ldots,\tau$, then $z_0$ will be a pole of $H(z)$ of multiplicity at most $tn+1$. Next suppose $z_1$ be any pole of $f(z)$ of multiplicity $t_0$ and a pole of at least one $f(z +jc)$, $j=1,2,\ldots,\tau$, of multiplicity $t_{j}\geq 0$. Then $z_1$ may or may not be a pole of $L_c(z,f)$. 
	From the above arguments and our assumption, we conclude that
	\bea\label{e3.11} N(r,H)&\leq& \ol N\left(r,\frac{1}{p(z)}\right)+N(r,(L_c(z,f))^n)+\ol N(r,L_c(z,f))+S(r,f)\nonumber\\&\leq& nN(r,L_c(z,f))+(\tau+1)\ol N(r,f)+S(r,f). \eea
	\par Next, we turn our attention towards the poles of $G(z)$. We know from (\ref{e3.3}) and (\ref{e3.10}) that the poles of $G(z)$ are at the zeros of $p(z)$ and $f(z)$ and at the poles of $f(z)$, $f(z +jc)$, $j=1,2,\ldots,\tau$. If $z_0$ is a zero of $p(z)$, zero of $f(z)$, or pole of $f(z + jc)$, $j=1,2,\ldots,\tau$, then  by (\ref{e3.10}) we know $z_0$ will be  at most a simple pole of $ G(z)$. If $z_0$ is a pole of $f(z)$ but not a pole of $f(z + jc)$, $j=1,2,\ldots,\tau$, then by the Laurent expansion of $G(z)$ at $z_0$, we obtain that $G(z)$ is analytic at $z_0$. Therefore, from our assumption and the discussions above, we know \bea\label{e3.12} N(r,G) &\leq& \ol N\left(r, \frac{1}{p(z)}\right) + \ol N(r,L_c(z,f)) + \ol N\left(r, \frac{1}{f}\right) + S(r,f) \nonumber\\ &\leq& \ol N(r,L_c(z,f)) + \ol N\left(r, \frac{1}{f}\right) + S(r,f). \eea \par
	Using {\it Lemma \ref{l8}}, from equations (\ref{e3.8}), (\ref{e3.11}) and (\ref{e3.12}) we have \beas &&mT(r,f)\\ &\leq& nm(r,f) + nN(r,L_c(z,f)) +(\tau+1)\ol N(r,f) + \ol N(r,L_c(z,f)) +\ol N\left(r, \frac{1}{f}\right ) + S(r,f)\\& \leq& nm(r,f)+n(\tau+1)N(r,f)+(\tau+1)\ol N(r,f)+(\tau+1)\ol N(r,f)+\ol N\left(r,\frac{1}{f}\right)+S(r,f)\\&\leq& \{(\tau+1)(n+2)+1\}T(r,f)+S(r,f),\eeas which contradicts the assumption that $m\geq (\tau+1)(n+2)+2$.
	This completes the proof of the theorem.
\end{proof}

\begin{center} {\bf Acknowledgement} \end{center} 
The authors wish to thank the referee for his/her valuable suggestions towards the improvement of the paper. The second author is thankful to University Grant Commission (UGC), Govt. of India for financial suport under UGC-Ref. No.: 1174/(CSIR-UGC NET DEC. 2017) dated 21/01/2019.

\end{document}